\RequirePackage{fix-cm}
\documentclass[smallextended]{svjour3}       
\smartqed  
\usepackage{graphicx}

\usepackage[utf8]{inputenc}
\usepackage[colorlinks,citecolor=blue]{hyperref}
\usepackage{cite}
\usepackage{textgreek}
\usepackage{tikz,pgfplots}
\usepackage[ruled,linesnumbered]{algorithm2e} 
\usepackage{siunitx}
\usepackage[commentmarkup=uwave,final]{changes}
\usepackage{tabularray}   
\usepackage{amsmath,amsfonts,amssymb}

\definechangesauthor[color=blue]{JO}
\definechangesauthor[color=purple]{TC}

\pgfplotsset{compat=newest}
\usetikzlibrary{plotmarks}
\usetikzlibrary{arrows.meta}
\usepgfplotslibrary{patchplots}
 
\newlength\figH
\newlength\figW
\DeclareMathOperator*{\interior}{int}
\newcommand{\Uset}{\mathcal{U}}

\spnewtheorem{assumption}{Assumption}{\bf}{\rm}

\newcommand{\ip}[1]{\langle#1\rangle}

\newcommand{\fundingname}{Funding}

\newenvironment{funding}{%
  \par\addvspace{17pt}\small\rmfamily
  \trivlist
  \ifx\fundingname\empty
    \item[]
  \else
    \item[\hskip\labelsep \bfseries \fundingname]
  \fi
}{%
  \endtrivlist\addvspace{6pt}
}

\newcommand{\dataavailname}{Data Availability}

\newenvironment{dataavail}{%
  \par\addvspace{17pt}\small\rmfamily
  \trivlist
  \ifx\dataavailname\empty
    \item[]
  \else
    \item[\hskip\labelsep \bfseries \dataavailname]
  \fi
}{%
  \endtrivlist\addvspace{6pt}
}

\newcommand{\sourcodename}{Source Code}

\newenvironment{sourcode}{%
  \par\addvspace{17pt}\small\rmfamily
  \trivlist
  \ifx\sourcodename\empty
    \item[]
  \else
    \item[\hskip\labelsep \bfseries \sourcodename]
  \fi
}{%
  \endtrivlist\addvspace{6pt}
}

\newcommand{\confintname}{Conflict of Interest}

\newenvironment{confint}{%
  \par\addvspace{17pt}\small\rmfamily
  \trivlist
  \ifx\confintname\empty
    \item[]
  \else
    \item[\hskip\labelsep \bfseries \confintname]
  \fi
}{%
  \endtrivlist\addvspace{6pt}
}

\journalname{Mathematical Programming Computation}
\begin{document}

\title{On the Practical Implementation of a Sequential Quadratic Programming Algorithm for Nonconvex Sum-of-squares Problems
}

\titlerunning{Sequential Quadratic Programming for Nonconvex Sum-of-squares Problems}        

\author{Jan Olucak        \and
        Torbj\o rn Cunis 
}


\institute{J. Olucak  \at
              Institute for Flight Mechanics and Controls, University of Stuttgart \\
              \email{jan.olucak@ifr.uni-stuttgart.de}           
           \and
           T. Cunis  \at
              Institute for Flight Mechanics and Controls, University of Stuttgart \\
              \email{tcunis@ifr.uni-stuttgart.de}  
}

\date{Received: date / Accepted: date}

\maketitle

\begin{abstract}
Sum-of-squares (SOS) optimization provides a computationally \newline tractable framework for certifying polynomial nonnegativity. If the considered problem is convex, the SOS problem can be transcribed into and solved by semi-definite programs. However, in case of nonconvex problems iterative procedures are needed. Yet tractable and efficient solution methods are still lacking, limiting their application, for instance, in control engineering. To address this gap, we propose a filter line search algorithm that solves a sequence of quadratic subproblems. Numerical benchmarks demonstrate that the algorithm can significantly reduce the number of iterations, resulting in a substantial decrease in computation time compared to established methods for nonconvex SOS programs. An open-source implementation of the algorithm along with the numerical benchmarks is provided.
\keywords{Numerical algorithms \and Optimization algorithms \and Sum-of-squares optimization \and Constraint control \and Lyapunov methods}
\subclass{90C30 \and 90C55 \and 93-08}
\end{abstract}

\section{Introduction}
Checking global nonnegativity of a function of several variables is a fundamental question that arises 
in many areas of applied mathematics. A tractable way to answer this question lies in polynomials that can be written as a sum of squares, a strict subset of the nonnegative polynomials. In~\cite{parillo2003} it was shown
that convex optimization over sum-of-squares (SOS) polynomials can be transcribed into and solved by semi-definite programs (SDPs). SOS programming problems arise from polynomial optimization~\cite{blekherman2013}, robust and stochastic optimization~\cite{bertsimas2011}, statistics and machine learning~\cite{magnani2005,barak2015}, and control engineering~\cite{jarvis-wloszek2003,jarvis-wloszek2005}, to name a few domains and applications.  


Depending on the problem size, this powerful tool comes with potentially high computational efforts.
Over the past decades substantial improvement in computational performance can be observed for convex SOS problems, not only due to the advances in computing hardware. This includes improved conic solvers~\cite{mosek,Clarabel_2024,scs,operations,han2024}, specialized SOS solvers~\cite{papp2022}, use of relaxations to (scaled) diagonally dominant cones~\cite{ahmadiMajumdar2017}, anytime feasible algorithms~\cite{anysos}, efficient implementations of polynomial operations and the reduction of a SOS program to an SDP \cite{jagt_efficient_2022,legat2017,Cunis2025}, or exploiting symmetry~\cite{loefberg2009} and sparsity~\cite{ahmadiEtAl2017} in the problem structure. 

\subsection{Nonconvex Sum-of-Squares Problems}
For nonconvex SOS problems less advancement can be observed. In the context of systems and control theory,  many practical engineering problems -- such as stability, invariance, and controllability -- typically hold only over local regions of the state space~\cite{chakrabortyEtAl2011,chakrabortyEtAl2011ii,iannelli2018,seilerBalas2010,yinEtAl2019,Yin2021,jarvisEtAl2003,Cunis2021aut,majumdarEtAl2013,lin2022,newton2025a,amesControlBarrierFunctions2019,clark2021,tan2006,schneeberger_sos_2023,loureiro2025}. Verifying these properties within such local regions leads to nonconvex SOS optimization problems, which prohibits the direct usage of (convex) SDPs. 
While nonconvex problem might also arise in other domains, they seem to have attracted little interest, possibly due to the hitherto lack of suitable solvers.

State-of-the art methods for nonconvex problems are {\em iterative} schemes \cite{chakrabortyEtAl2011,chakrabortyEtAl2011ii,iannelli2018,seilerBalas2010,yinEtAl2019,Yin2021,jarvisEtAl2003,Cunis2021aut,majumdarEtAl2013,lin2022,newton2025a,amesControlBarrierFunctions2019,clark2021,tan2006,schneeberger_sos_2023,loureiro2025}. These schemes include bisection for {quasiconvex} problems (see \cite{seilerBalas2010} for details), coordinate descent (also referred to as alternating-direction) with convex subproblems, and hybrid approaches combining coordinate descent with bisections for quasiconvex subproblems. Table~\ref{tab: overviewCtrlApp} lists different control applications and the iterative methods applied.
Coordinate descent iteratively optimizes over one variable at a time while keeping others fixed, simplifying nonconvex problems into sequences of convex SOS programs. However, unless the problem is quasiconvex, these techniques lack convergence guarantees, may require many iterations with unpredictable performance, and require a feasible initial guess. Additionally, more effort is needed by the user to split the nonconvex problem into the convex subproblems.
There also exist nonlinear SDP solvers such as {\tt PENLAB} \cite{fiala2013}, yet their development appears to have stalled and hence is not further considered. 

\begin{table}[h!]
    \caption{Overview of control applications with their solution method using coordinate-descent with convex subproblems or a hybrid approach with quasiconvex subproblems.}
    \label{tab: overviewCtrlApp}
    \centering
    \begin{tabular}{p{5cm} l c c c}
    \hline
    \hline
    Application & Ref. & hybrid & convex \\
        \hline
    Region-of-attraction analysis  & \cite{chakrabortyEtAl2011,chakrabortyEtAl2011ii,iannelli2018,loureiro2025} & \checkmark & \checkmark  \\
    Reachability analysis        &  \cite{yinEtAl2019,Yin2021,jarvisEtAl2003} &\checkmark &\checkmark  \\
                         & \cite{Cunis2021aut,majumdarEtAl2013,lin2022} &  &  \checkmark\\
    Feedback law synthesis    & \cite{jarvisEtAl2003,newton2025a} & \checkmark &\checkmark  \\
    Control barrier function synthesis  & \cite{amesControlBarrierFunctions2019,clark2021} & & \checkmark  \\
    Local control Lyapunov function synthesis & \cite{tan2006}  & \checkmark &\checkmark  \\ 
    Control barrier \& control Lyapunov function synthesis  & \cite{schneeberger_sos_2023} & & \checkmark  \\
    Nonlinear observer design & \cite{tan2006}  & & \checkmark  \\
        \hline
    \hline
    \end{tabular}
\end{table}

\subsection{Challenges for a Practical Implementation}
Inspired by sequential convex programming, we proposed sequential (linear) SOS programming~\cite{Cunis2023acc}. This method reduced the number of iterations compared to coordinate descent and provides (local, linear) convergence guarantees as an instance of quasi-Newton methods.
However, it lacks a comprehensive globalization strategy and it does not explicitly address constraint violations\deleted[id=TC]{, and it lacks an available efficient implementation for practical usage}. \added[id=JO]{Furthermore, for practical problems faster convergence is desirable. Thus, leveraging a sequential quadratic, instead of sequential linear programming approach with a suitable globalization strategy is desirable.} 
Efficient constraint violation checks are crucial, particularly in nonlinear optimization \added[id=TC]{used for system analysis and verification, to ensure correctness of the results}. While simple function evaluations in combination with some norm are used in parameter optimization, conic and SOS programming require projections onto the corresponding cone(s)~\cite{Boyd2005}. As no analytical projection operator is available for the SOS cone, projections must be computed using reliable numerical methods.
Globalization strategies are critical for ensuring global convergence of Newton-type methods~\cite[Chapter~1]{betts2010}.
Merit functions, such as $l_p$-norm penalty or augmented Lagrangian, \added[id=JO]{in combination with line search,} balance cost reduction and constraint violation but require careful penalty parameter selection.  This can introduce additional computational costs due to the need for iterative penalty adjustments (cf.~\cite{nocedalNumericalOptimization2006,bieglerNonlinearProgrammingConcepts2010}) and introduce numerical instability if wrongly selected. This issue is especially relevant in the context of SOS constraint violations as frequent checks, needed for the penalty adjustment, may significantly increase the computational burden. 
The filter algorithm~\cite{fletcherNonlinearProgrammingPenalty2002} offers a promising alternative by decoupling feasibility and optimality. It tracks objective values and constraint violations, making it robust to penalty parameter \added[id=TC]{selection} and gives more flexibility in computing the next solution candidate~\cite{bieglerNonlinearProgrammingConcepts2010}. Unlike Merit functions, the filter explicitly handles infeasibility through a {restoration phase}~\cite[Section~3.3]{fletcherNonlinearProgrammingPenalty2002}, alleviating numerical instability and ensuring global convergence at the cost of a higher implementation effort.  \added[id=JO]{Unlike line search, trust-region methods modify both the step-length and search direction  using either Merit functions or a filter. According to~\cite[Section~5.7]{bieglerNonlinearProgrammingConcepts2010} they offer superior convergence properties (compared to line search), yet, they come at a higher computational cost and lead to larger problems due to additional (norm) constraints.}
Given these considerations, a filter line search algorithm appears to be a suitable candidate for \added[id=TC]{nonconvex SOS programs}.

\subsection{Contribution and Structure}
Based on the \added[id=TC]{considerations} in the previous sections, we propose a filter line search approach \added[id=TC]{to sequential quadratic SOS programming} which significantly enhances the method introduced in \cite{Cunis2023acc}. Our work builds upon the filter line search algorithm established in \cite{wachterLineSearchFilter2005} while takking inspiration from \cite{wachterImplementationInteriorpointFilter2006a}. Whereas  global convergence for an active set method is proven in~\cite{wachterLineSearchFilter2005}, the extensions to SOS is not straightforward. We leave the global convergence analysis for future work, concentrating on the practical aspects and providing a local convergence proof.
The key contributions of this paper are as follows:
\begin{itemize}
    \item {Convergence Analysis:} A local convergence analysis for the sequential SOS algorithm with quadratic subproblems is provided.
    \item {Design Aspects:}   SOS-specific design aspects for nonlinear programming are investigated and discussed including the efficient constraint violation check and the design of a feasibility restoration phase.
    \item  {Practical Implementation:} An implementation of the proposed filter line search algorithm is  provided as part of the \added[id=TC]{open-source} software CaΣoS~\cite{Cunis2025a}.
\end{itemize}
We demonstrate the effectiveness and practical applicability of the proposed approach \added[id=TC]{in several benchmark problems inspired from real-world control engineering examples}. Additionally, we compare it to state-of-the-art methods for nonconvex SOS programming.

This paper is structured as follows: In Section~\ref{sec: ProblemStatement} the problem statement of nonconvex SOS problems is motivated and explained. A proof for local convergence and the background on the filter line search algorithm are provided in Section~\ref{sec: Methodology}. The details regarding the practical implementation are given in Section~\ref{sec: PracticalImplementation}. In Section~\ref{sec: Benchmarks} the proposed algorithm is benchmarked against the state-of-the art solution method for nonconvex SOS programs. Practical aspects are discussed at the end.

\section{Problem Statement}
\label{sec: ProblemStatement}
\subsection{Sum-of-squares Polynomials}
Let $x = (x_1, \ldots, x_n)$ be a tuple of indeterminate variables. A polynomial $p$ in $x$ up to degree $d$ is a linear combination
\begin{align*}
    p = \sum_{\| \alpha \|_1 \leq d} c_\alpha x^\alpha
\end{align*}
with multi-indices $\alpha = (\alpha_1, \ldots, \alpha_n) \in \mathbb N_0^n$,
where $x^\alpha = x_1^{\alpha_1} \cdots x_n^{\alpha_n}$ and $\| \alpha \|_1 = \sum_i \alpha_i$. Denote by $\mathbb{R}_d[x]$ the set of polynomials in $x$ with real coefficients $c_\alpha \in \mathbb R$ up to degree $d$.

\begin{definition}
    A polynomial $p \in \mathbb R_{2d}[x]$ is {a} {\em sum-of-squares} {polynomial} ($p \in \Sigma_{2d}[x]$) if and only if there exist {$m \in \mathbb N$ and} $p_1, \ldots, p_m \in \mathbb R_d[x]$ such that $p = \sum_{i=1}^m (p_i)^2$.
\end{definition}

The cone of sum-of-squares polynomials $\Sigma_d[x]$ forms a convex cone in $\mathbb R_d[x]$ with dual cone $\Sigma_d[x]^*$, which
is isometric to the cone of sum-of-squares polynomials \cite{Lasserre2001}.
Given $d \in \mathbb N$, denote by $s_d\in \mathbb{N}$ the number of monomials in $x$ up to degree $d$. A polynomial $p \in \mathbb R_{2d}[x]$ is sum-of-squares if and only if there exists a \added[id=TC]{positive semidefinite} matrix $Q \in \mathbb R^{s_d \times s_d}$ satisfying $p = \zeta^\top Q \zeta$, where $\zeta \in \mathbb R_d[x]^{s_d}$ is the vector of monomials up to degree $d$\deleted[id=TC, comment={Appears not used hereafter}]{, and $\mathbb S_{s_d}^+$ denotes the set of positive semidefinite matrices}.

\subsection{Nonconvex Sum-of-squares Optimization}
We are interested in solving general nonconvex SOS problems of the form
\begin{align}
    \label{eq:sos-nonlinear}
    \min_{\xi \in \mathbb R_{2d}[x]^n} f(\xi) \quad \text{s.t. $\xi \in \Sigma_{2d}[x]^n$ and $g(\xi) \in \Sigma_{2d'}[x]^m$}
\end{align}
where $f: \mathbb R_{2d}[x]^n \to \mathbb R$ and $g: \mathbb R_{2d}[x]^n \to \mathbb R_{2d'}[x]^m$ are \added[id=TC]{(twice)} differentiable functions. Since the degrees $d, d' \in \mathbb N_0$ are fixed, $f$ and $g$ are mappings between finite vector spaces and the gradients $\nabla f(\cdot): \mathbb R_{2d}[x]^n \to \mathbb R$ and $\nabla g(\cdot): \mathbb R_{2d}[x]^n \to \mathbb R_{2d'}[x]^m$ are finite linear operators. \added[id=JO]{Define the Lagrangian as ${L(\xi,\lambda) =  f(\xi)  - \langle \lambda, g(\xi) \rangle}$ with Lagrange variable $\lambda \in (\Sigma_{2d'}[x]^m)^*$.}
Under a suitable constraint qualification, the optimal solution $\xi_\star$ of \eqref{eq:sos-nonlinear} satisfies the Karush--Kuhn--Tucker (KKT) necessary conditions
\begin{align}
    \label{eq:kkt-nonlinear}
    \left\{
    \begin{alignedat}{2}
        &\nabla f(\xi_\star) - \nabla g&(\xi_\star)^* \lambda_\star - \mu_\star = 0& \\
        &\langle \mu_\star, \xi_\star \rangle = 0, & \langle \lambda_\star, g(\xi_\star) \rangle = 0& \\
        &\xi_\star \in \Sigma_{2d}[x]^n, & g(\xi_\star) \in \Sigma_{2d'}[x]^m\!&
    \end{alignedat}
    \right.
\end{align}
where $\mu_\star \in (\Sigma_{2d}[x]^n)^*$ is a normal vector.
Eq.~\eqref{eq:kkt-nonlinear} can compactly be written as the generalized equation
\begin{align*}
    \varphi(\xi_\star,\lambda_\star) + \mathcal N(\xi_\star, \lambda_\star) \ni 0
\end{align*}
with
\begin{align*}
    \varphi: (\xi, \lambda) \mapsto (\nabla f(\xi) - \nabla g(\xi)^* \lambda, \, g(\xi))
\end{align*}
where $\mathcal N(\xi, \lambda)$ denotes the normal cone to $\Sigma_{2d}[x]^n \times (\Sigma_{2d'}[x]^m)^*$ at $(\xi, \lambda)$.

\subsection{Motivational Example}
We give a practical problem arising in control engineering. Assume a continuous-time nonlinear dynamic system defined by a polynomial vector field $f: \mathbb R^n \times \mathbb R^m \rightarrow \mathbb R[x]^n$, where $x \in \mathbb R^n$ denotes the state vector and $u \in \mathbb R^m$ is the control input, satisfying $f(0,0) = 0$.  We are interested to synthesize a (linear) control law $\kappa(x)$ which asymptotically stabilizes the system's origin in a domain $\mathcal X = \{x \in \mathbb R^n \mid g(x) \leq 0\}$, i.e., we are interested in a non-conservative state-constrained region-of-attraction. We therefore want to maximize the volume of the sublevel set of the ROA while fulfilling two set-containment constraints, i.e., 
\begin{subequations}
    \label{eq: motEx}
\begin{align}
\max_{V, \kappa} \quad & \mathrm{vol}(V,\gamma) \\
\textrm{s.t.} \quad &    \{ x \in \mathbb R^n \mid V(x) \leq \gamma \} \subseteq \{x \in \mathbb R^n \mid g(x) \leq 0\} \\
   & \{ x \in \mathbb R^n \mid V(x) \leq \gamma \} \subseteq  \{x \in \mathbb R^n \mid   \langle \nabla V(x), f_\kappa(x)\rangle < 0 \}
\end{align}
\end{subequations}
where $\mathrm{vol}(V,\gamma)$ is the Lebesque measure of the $\gamma$-sublevel set of $V$, $f_\kappa = f(\cdot,\kappa(\cdot))$ are the closed-loop polynomial dynamics, $\langle \cdot , \cdot \rangle$ denotes the inner product, $V > 0$ is the Lyapunov function, and $\gamma$ is a stable level of~$V$. 

To transcribe the high-level problem \eqref{eq: motEx} into the SOS problem \eqref{eq:sos-nonlinear}, we make use of the so-called generalized S-procedure, a special case of the Positivstellensatz.

\begin{theorem}[Generalized S-procedure \cite{tan2006}]
    \label{Theo: genSprocedure}
    Given $p_0, p_1, \ldots, p_N \in \mathbb{R}[x]$, if there exist $s_1, \ldots, s_N \in \Sigma[x]$ such that $p_0- \sum_{k=1}^N s_k p_k \in \Sigma[x]$, then
	\begin{align}
		\bigcap_{k=1}^N \{x \in \mathbb{R}^n \mid p_k(x) \geq 0\} \subseteq \{x \in \mathbb{R}^n \mid p_0(x) \geq 0\} \nonumber
	\end{align}
    holds.
\end{theorem}

Applying Theorem~\ref{Theo: genSprocedure} to the motivational example \eqref{eq: motEx} yields
\begin{equation*}
\begin{aligned}
\max_{\substack{V, \kappa \in \mathbb R[x] \\ s_1, s_2 \in \Sigma[x]}} \quad & \mathrm{vol}(V,\gamma)\\
\textrm{s.t.} \quad &      s_1(V-\gamma) - g(x) &\in \Sigma[x] \\
    &s_2(V-\gamma) - \langle \nabla V, f_\kappa \rangle - \epsilon(x) &\in \Sigma[x] \\
    &V -   \epsilon(x) &\in \Sigma[x]
\end{aligned}
\end{equation*}
where $s_1,s_2$ are SOS multipliers resulting from Theorem~\ref{Theo: genSprocedure} and $\epsilon(x)$ is a small term to ensure stric inequalities are met. 
One can observe that the problem is bilinear in the decision variables $ s_1,s_2,V$ and $\kappa$ is entering nonlinearly into the second constraint. Clearly, this problem cannot be directly transcribed into and solved by an SDP.

Statements of forward invariance, asymptotic stability, regions of attraction, passivity, bounds on the $L_2$-gain, input-to-state stability, control Lyapunov and barrier functions, and many more properties of nonlinear control systems can be written as dissipation inequalities~\cite{ebenbauerAllgoewer2006} of the form
\begin{equation*}
    \nabla \phi(x) f(x, u) \leq \varsigma(x,u) \quad \forall (x, u) \in \Omega \times \mathcal U
\end{equation*}
where $\phi$ and $\varsigma$ are the so-called storage function and supply rate, $f:\mathbb R^n \times \mathbb R^m \rightarrow \mathbb R^n$ are the system dynamics, and $\Omega \subset \mathbb R^n$ and $\mathcal U \subset \mathbb R^m$ are subsets of the state and input space, respectively. Sum-of-squares constraints provide systematic formulations of dissipation inequalities for polynomial dynamical systems and play an important role for the benchmarks in Section~\ref{sec: Benchmarks}.


\section{Methodology}
\label{sec: Methodology}
This section provides the general concept and core idea of the proposed approach. We start with the basic principle behind sequential quadratic (SOS) programming. Unlike~\cite{Cunis2023acc}, we consider quadratic instead of linear subproblems. A local convergence analysis for quadratic problems is provided. Finally, the concept of filter line search as globalization strategy is introduced which serves as basis for the practical implementation.

\subsection{Sequential Quadratic SOS Programming}
We solve the nonlinear optimization problem \eqref{eq:sos-nonlinear} by a sequence of local {\em convex} sum-of-squares problems relative to a previous candidate solution (or initial guess) $\xi^k$. The solution of the local problem is then used to update the solution candidate and another local problem is solved, as shown in Alg.~\ref{alg: seqLinProg}.
This framework is a standard approach in nonlinear optimization~\cite{nocedalNumericalOptimization2006}.

\begin{algorithm}[ht!]
    \caption{Sequential SOS Algorithm}\label{alg: seqLinProg}
    \KwIn{Initial guess $\xi^0$}
        \For{$k = 0, 1, \ldots, N^\mathrm{max}$}
        {
            Solve convex SOS problem, parametrized in $\xi^k$, for search direction $\omega^k$ \\
            Perform line search along $\xi^k + \alpha \, \omega^k$ for step length $\alpha^k$ \\
            Compute correction or restore feasibility, if necessary \\
            Set new solution candidate to $\xi^{k+1} := \xi^k + \alpha^k \, \omega^k$ \\
            \If{termination criteria is satisfied}{\Return}
            Update Hessian approximation
        } 
\end{algorithm}

At the iterate $\xi^k$, we solve the quadratic problem 
	\begin{subequations}
    \label{eq:QsubProb}
    \begin{align}
\min_{\xi \in \mathbb R_{2d}[x]^n} \tfrac{1}{2}& \langle H_k (\xi - \xi^k), \xi - \xi^k \rangle + \langle \nabla f(\xi^k), \xi - \xi^k \rangle\\
\textrm{s.t.}& \quad       g(\xi^k) + \nabla g(\xi^k)(\xi - \xi^k) \in \Sigma_{2d'}[x]^m \label{eq: linConQsub}\\
   &\quad \xi \in \Sigma_{2d}[x]^n
\end{align}
\end{subequations}
where $H_k: \mathbb R_{2d}[x]^n \to (\mathbb R_{2d}[x]^n)^*$ is an approximation of the Hessian of the Lagrangian in \eqref{eq:kkt-nonlinear}. Some conic solvers \cite{scs,Clarabel_2024} can handle a quadratic cost directly, otherwise one rewrites the quadratic cost into a second-order cone constraint using the epigraph reformulation~\cite{Boyd2005}. In our implementation, we rely on existing SDP solvers.

The KKT conditions for the quadratic subproblem~\eqref{eq:QsubProb} read
\begin{align}
    \label{eq:kkt-quadratic}
    \left\{
    \begin{alignedat}{2}
        & H_k(\xi_+ - \xi^k) + \nabla f(\xi^k) - \nabla g(\xi^k)^* \lambda_+ - \mu_+ = 0 \\
        &\langle \mu_+, \xi_+ \rangle = 0,  \langle \lambda_+, g(\xi_+) +\nabla g(\xi^k)(\xi_+ - \xi^k)  \rangle = 0 \\
        &\xi_+ \in \Sigma_{2d}[x]^n, \quad g(\xi) +\nabla g(\xi^k)(\xi_+ - \xi^k) \in \Sigma_{2d'}[x]^m&
    \end{alignedat}
    \right.
\end{align}
where $\xi_+$, $\lambda_+$ and  $\mu_+$ denote the optimal solution, corresponding Lagrange multipliers, and normal vector of~\eqref{eq:QsubProb}, respectively. \added[id=TC]{The (primal) search direction in Alg.~\ref{alg: seqLinProg} is then given as $\omega^k = \xi_+ - \xi^k$.}
This approach is also known as {\em quasi-Newton} method \cite{Izmailov2014} and it exhibits fast local convergence with a well-chosen initial guess; however, it often requires a globalization technique for robust performance from arbitrary starting points. We are going to utilize a line search to compute
the new candidate solution as
\begin{subequations}
    \label{eq: newTrialPoint}
\begin{alignat}{2}
    \xi^{k+1} &= \xi^k & &+ \alpha^k {(\xi_+-\xi^k)}  \\
    \lambda^{k+1} &= \lambda^k & &+ \alpha^k (\lambda_+ - \lambda^k)
\end{alignat}
\end{subequations}
where $\alpha^k \in (0, 1]$ is the chosen step length at iteration $k$. Approximation methods for the Hessian are discussed in Subsection~\ref{subsubsec: HessianApprox}.

\subsection{Local Convergence Analysis}

We provide a brief, local convergence analysis of the sequential algorithm without line search ($\alpha^k \equiv 1$). If $(\xi_+, \lambda_+)$ solves the KKT conditions \eqref{eq:kkt-quadratic} with $\mu_+$, the update \eqref{eq: newTrialPoint} can be written as \begin{align}
    \label{eq:newton-update}
    (\xi^{k+1}, \lambda^{k+1}) \in (\xi^k, \lambda^k) - (\Phi(\xi^k, H_k) + \mathcal N)^{-1} \big[\varphi(\xi^k, \lambda^k)\big]
\end{align}
where $(\cdot)^{-1}$ is the (possibly multi-valued) inverse mapping, with
\begin{align*}
    \Phi(\xi^k,H_k): (\xi, \lambda) &\mapsto (H_k \xi - \nabla g(\xi^k)^* \lambda, \nabla g(\xi^k) \xi)
\end{align*}
and candidate solution $\xi^k$ and Hessian approximation $H_k$. The key observation for the convergence analysis is that \eqref{eq:newton-update} constitutes a set-valued Newton update step for the nonlinear KKT conditions \eqref{eq:kkt-nonlinear}, where $\Phi(\xi^k, H_k)$ is an approximation of the first-order derivative of $\varphi$. 

The convergence properties of set-valued Newton methods are a long studied subject of variational analysis \cite{Dokov1998, Dontchev2013, mordukhovich_generalized_2021, Louzeiro2023}. Local convergence can usually be established provided that the inverse of $(\Phi(\xi^k, H_k) + \mathcal N)$ in \eqref{eq:newton-update} is well defined and the approximation $\Phi(\xi^k, H_k)$ is `sufficiently good' -- the better the approximation, the faster the rate of convergence. In particular, we make the following assumptions.

\pagebreak
\begin{assumption}[Regularity]
    \label{ass:newton}
    The exist constants $\varkappa, \ell_k > 0$, $k \in \mathbb N$, such that\nopagebreak
    \begin{enumerate}
        \item[(a)] the inverse $(\Phi(\xi^k, H_k) + \mathcal N)^{-1}$ is locally single-valued around $(\xi_\star, \lambda_\star)$ and locally Lipschitz continuous at $0$ with constant $\varkappa$ uniformly for $\xi^k$ around $\xi_\star$ and all $H_k$;
        \item[(b)] the mapping $\varphi(\xi^k, \lambda^k) + \Phi(\xi^k, H_k)(\xi' - \xi^k, \lambda' - \lambda^k)$ is locally Lipschitz continuous relative to $(\xi^k, \lambda^k)$ at $(\xi_\star, \lambda_\star)$ with constant $\ell_k$ uniformly for $(\xi', \lambda')$ around $(\xi_\star, \lambda_\star)$ and all $H_k$
    \end{enumerate}
    for all $k \in \mathbb N$.
\end{assumption}

Part (a) of Assumption~\ref{ass:newton} is called {\em strong regularity} of the set-valued mapping $(\Phi(\xi^k, H_k) + \mathcal N)$ and corresponds to unique primal-dual solutions to the convex problem \eqref{eq:QsubProb} that are continuous under first-order perturbations \added[id=TC]{\cite{Dinh2010}}. 
Regularity can be enforced by choosing a positive definite approximation $H_k$. The Lipschitz constant in part (b) measures how good the linear operator $\Phi(\xi^k, H_k)$ describes the {\em change} of $\varphi(\xi^k, \lambda^k)$ as $(\xi^k, \lambda^k) \to (\xi_\star, \lambda_\star)$ and can be thought of as generalization of an error in the derivative of $\varphi$. In particular, Assumption~\ref{ass:newton}(b) is satisfied with $\ell_k = 0$ if $H_k$ is the second-order derivative $\nabla_{\xi\xi}^2 [f(\xi_k) - \langle \lambda_k, g(\xi_k) \rangle]$.

\begin{theorem}[Local convergence]
    Suppose that Assumption~\ref{ass:newton} holds at a solution $(\xi_\star, \lambda_\star)$ to \eqref{eq:kkt-nonlinear}; then the sequence generated by \eqref{eq:newton-update} is unique and converges
    \begin{enumerate}
        \item linearly, if $\ell_k < \varkappa^{-1}$ for all $k \in \mathbb N$;
        \item superlinearly, if $\ell_k \searrow 0$ as $k \to \infty$;
        \item quadratically, if $\ell_k \equiv 0$;
    \end{enumerate}
    to $(\xi_\star, \lambda_\star)$ provided that $(\xi^0, \lambda^0)$ is sufficiently close.
\end{theorem}
\begin{proof}
    By virtue of \cite[Theorem~8.5]{Dontchev2021} with $h: (\xi,\lambda) \mapsto \varphi(\xi_\star, \lambda_\star) + \Phi(\xi_\star, H_k)(\xi - \xi_\star, \lambda - \lambda_\star)$, the sequence $\{(\xi^k, \lambda^k)\}_{k \in \mathbb N}$ generated by \eqref{eq:newton-update} is unique and satisfies
    \begin{align*}
        \| (\xi^{k+1}, \lambda^{k+1}) - (\xi_\star, \lambda_\star) \| \leq \varkappa \ell_k \| (\xi^k, \lambda^k) - (\xi_\star, \lambda_\star) \|
    \end{align*}
    for all $k \in \mathbb N$, provided that $(\xi^k, \lambda^k)$ is close to $(\xi_\star, \lambda_\star)$; this implies (super) linear convergence if $\varkappa \ell_k < 1$ (resp., $\varkappa \ell_k \searrow 0$), and quadratic convergence if $\varkappa \ell_k \equiv 0$. \qed
\end{proof}

\subsection{Line Search with Filter}
\label{subsec: FilterLinesearch}
This subsection provides  the core components of the filter line search algorithm, which is primarily based on~\cite[p. 120]{bieglerNonlinearProgrammingConcepts2010}. We outline the core parts to make this paper self-contained. The \added[id=TC]{filter line search algorithm is} summarized in Algorithm~\ref{alg: filterLineSearch}. \added[id=JO]{Three major steps are conducted: 1) check filter acceptance; 2) check sufficient decrease in cost; and 3) augment filter if necessary.}

\subsubsection{Filter Acceptance}
Unlike merit functions, that combine cost $f$ and constraint violation $\theta$ \added[id=TC]{in a single objective}, a filter tries to deal with these two conflicting goals in a bi-objective optimization manner. No penalty parameter must be estimated, as is the case for penalty functions or augmented Lagrangian. \added[id=JO]{The filter $\mathcal{F}$ consists of $l \in \mathbb N$ so-called dominant pairs $(f(\xi^l),\theta(\xi^l)) \in \mathcal{F}$ which form a \textit{forbidden region}. In a first step, the current trial point is checked for filter acceptance, that is, checking if the trial point lies not in the forbidden region and makes sufficient progress regarding constraint violation.} A new trial point, \added[id=TC]{usually the result of \eqref{eq: newTrialPoint} for the candidate step length $\alpha$,} is accepted to the filter $\mathcal F$ if
\begin{align}
    f(\xi^{k+1}) < f(\xi^l) \quad \text{or}\quad \theta(\xi^{k+1}) < \theta(\xi^l), \quad  \text{for all $(f(\xi^l), \theta(\xi^l)) \in \mathcal{F}$}
    \label{eq: FilterAcceptCriteria}
\end{align}
Condition~\eqref{eq: FilterAcceptCriteria} is only an acceptance or rejection criterion, but does not necessarily ensure progress. Therefore, additional sufficient decrease conditions  are checked for an acceptable point, as outlined next.  If the candidate fails or violates the first condition~\eqref{eq: FilterAcceptCriteria}, a second-order correction may be applied, or the step length is reduced. 

\subsubsection{Sufficient Decrease Conditions}
The first condition ensures that the current search direction $d = \xi^{k+1} - \xi^k$ is a descent direction
\begin{equation}
    \nabla f(\xi^k)^\top d < 0 
    \label{eq: descentDirection}
\end{equation}
where $\nabla f(\xi^k)$ is the gradient of the original cost. A second condition is used and reads
\begin{equation}
    \alpha(-\nabla f(\xi^k)^\top d )^{s_\phi} > \delta \theta(\xi^{k})^{s_\theta}
    \label{eq: sufficientDecrease2}
\end{equation}
where $s_\phi, s_\theta \in (0,1)$ and $\delta \in (0,1)$ are small constants. This condition ensures a sufficiently large progress for the cost function compared to the current constraint violation~\cite{bieglerNonlinearProgrammingConcepts2010}. Equations~\eqref{eq: descentDirection} and \eqref{eq: sufficientDecrease2} are also referred as \textit{f-type} switching condition.

In case the \textit{f-type} switching is fulfilled and additionally if $\theta(\xi^{k+1}) < \theta_{\text{min}}$ \cite{wachterImplementationInteriorpointFilter2006a} then the Amijo condition is checked
\begin{equation}
    \phi(\xi^{k+1}) \leq \phi(\xi^k) + \alpha \rho \nabla \phi(\xi^k)^\top d
    \label{eq: AmijoCondition}
\end{equation}
where $\rho \in [0,\frac{1}{2}]$. \comment[id=TC]{Can $\rho$ be zero?}

If no reduction in cost is possible (according to \textit{f-type}) it is checked if either the constraint violation or cost is smaller than the previous iterate with a small envelope, \added[id=TC]{that is,}
\begin{equation}
    \theta(\xi^{k+1}) \leq \theta(\xi^k) (1- \gamma_\theta) \quad \text{or} \quad
    \phi(\xi^{k+1}) \leq \phi(\xi^k) - \gamma_f \theta(\xi^k) 
\label{eq: constrainProgress}
\end{equation}
where $\gamma_\theta \in [0,1]$ and $\gamma_f\in [0,1]$ are small constants. \comment[id=TC]{Can $\gamma$ be zero?}
\subsubsection{Filter Augmentation}
\added[id=JO]{Unlike~\cite{fletcherNonlinearProgrammingPenalty2002}, the filter is only augmented if, and only if, \added[id=TC]{at least one} of the sufficient decrease conditions~\eqref{eq: descentDirection}--\eqref{eq: AmijoCondition} does not hold. This stricter augmentation rule avoids cycling, \added[id=TC]{where} the iterates alternately improve one of the two measures but deteriorate the other \cite{wachterLineSearchFilter2005}.}

\subsubsection{Full Algorithm}
Algorithm~\ref{alg: filterLineSearch} outlines the filter line search procedure (cf.~~\cite{bieglerNonlinearProgrammingConcepts2010}).  To avoid the Maratos effect~\cite{nocedalNumericalOptimization2006}, we additionally add a second-order correction (details in Section~\ref{subsec:SOC}). If the second-order correction is able to find an acceptable trial point, the line search is terminated, otherwise, the step length is reduced.  In case the step-length falls below the minimum step-length $\alpha_\text{min}$ a feasibility restoration phase is invoked. The primary goal of the feasibility restoration phase is to minimize the constraint violation, such that the trial point is acceptable to the filter.

\begin{algorithm}[ht!]
    \caption{Filter line search with second-order correction}\label{alg: filterLineSearch}
    \KwIn{$\xi^k, \xi_+$}
    \KwOut{$\xi^{k+1}$}
    Set $\alpha = 1$\\
        \While{$\alpha \geq  \alpha_\mathrm{min} $}
        {
            Compute new trial point $\xi_\alpha = \alpha(\xi_+ - \xi^k)$\\
            \eIf{Acceptable to filter~\eqref{eq: FilterAcceptCriteria}} 
            {   
                \eIf{\eqref{eq: descentDirection} and~\eqref{eq: sufficientDecrease2} hold} 
                {
                    \eIf{\eqref{eq: AmijoCondition} holds} 
                    {
                        \Return{$\xi_\alpha$} \\
                    }   
                    {
                       Invoke SOC  if $\alpha = 1$\\
                    }
                }    
                {
                    \eIf{\eqref{eq: constrainProgress} holds} 
                    {
                        \Return{$\xi_\alpha$}
                    }   
                    {
                        Invoke SOC  if $\alpha = 1$\\
                    }
                }
            }   
            {
                Invoke SOC  if $\alpha = 1$\\
            }
            Adjust step length $\alpha = \frac{1}{2}\alpha$\\
        } 
        \If{[\eqref{eq: descentDirection} and~\eqref{eq: sufficientDecrease2}] or~\eqref{eq: AmijoCondition} do not hold} 
        {
            Augment filter
        }   
\end{algorithm}

\subsection{Required SOS Modification}
The filter frequently needs to check progress in cost and feasibility. While checking progress in cost is a simple function evaluation, evaluating the constraint violation (feasibility) for SOS constraints is more complicated. In parameter optimization, constraint violation can be assessed by evaluating the constraints at the current iterate and use some norm, e.g., \(\theta(\xi^k):= \|g(\xi^k)\|_\infty  \). For conic optimization, the constraint violation check involves a projection onto the corresponding cone  (cf.~\cite{Boyd2005}). Although some cones possess relatively simple (analytical) projection operators, this is not the case for the SOS cone. Since the constraint violation is frequently evaluated, we are interested in precise yet efficient methods for this task. Similarly, a dedicated feasibility restoration, which primary goal is to reduce constraint violation, must be designed and implemented for SOS.  

\section{Practical Implementation}
\label{sec: PracticalImplementation}
This section gives a detailed overview of the implementation of the proposed sequential quadratic SOS algorithm. Algorithm~\ref{alg: fullSQP} shows the complete optimization procedure. 
An open-source implementation is available in the recently introduced software CaΣoS~\cite{Cunis2025a}, which is specifically designed to deal with nonconvex SOS problems. Hence, we will shortly outline the core ideas behind CaΣoS. Afterwards, the SOS modifications are discussed.  

\subsection{Background CaΣoS Framework}
The MATLAB-based software CaΣoS  was presented in~\cite{Cunis2025a}. Its framework for nonlinear symbolic polynomial expressions is inspired by and built upon  CasADi~\cite{andersson2019}. It implements efficient polynomial data types, linear and nonlinear polynomial operations, and, with the CasADi backend, parameterized solvers and functions. The last feature -- novel to sum-of-squares toolboxes -- is advantageous for iterative procedures.

\begin{algorithm}[h!]
\caption{Sequential Quadratic SOS }\label{alg: fullSQP}
\KwIn{Initial guess $\xi^0$, parameters}
Initialize Hessian $H_0$, Initialize filter, $k = 0$\\
\While{$k \leq  N_\mathrm{max} $}{

    \eIf{$ ||\nabla L(\xi^{k},\lambda^k)|| \leq \epsilon_L $, and
         $|| \theta(\xi^{k})|| \leq \epsilon_\theta$ }
    {
        Terminate\\
    }
    {
    \text{Solve}~\eqref{eq:QsubProb} \text{to obtain} $\xi^{*}$ \text{and} $\lambda^{*}$\\
    \eIf{Feasible}
    {
    Call Algorithm \ref{alg: filterLineSearch} (filter line search)\\
    \eIf{$\alpha < \alpha_\mathrm{min}$}
    {
        Call Algorithm~\ref{alg: feasRes} (Feas. Restoration) to obtain $\xi_r$  \\
        \eIf{$\xi_r$ acceptable to filter}
        {   
            Set $\xi^k = \xi_r$\\
            Go to line 3 
        }
        {
            Terminate (locally infeasible)
        }
    }  
    {
        Accept trial point
    }

    }
    {
        Call Algorithm~\ref{alg: feasRes} (Feas. Restoration) to obtain $\xi_r$ \\
        \eIf{$\xi_r$ acceptable to filter}
        {   
            Set $\xi^k = \xi_r$\\
            Go to line 3
        }
        {
            Terminate (locally infeasible)
        }
    }

        Compute Hessian approximation\\
        set $k = k+1$ 
        }
 
}
\end{algorithm}

\subsection{Constraint Violation Estimation}
\label{subsec: conVioCheck}
Unlike standard optimization, the constraint violation check for SOS constraints is not a simple function evaluation.
In the following, different methods are discussed and compared in a small benchmark.

We assume $n_g$ nonlinear constraints. The constraint violation $\theta(\xi^k)$ for the solution $\xi^k$ indicates whether the constraints $g(\xi^k) \in \Sigma[x]^{n_g}$ are satisfied ($\theta(\xi^k) = 0$) or, if not, measures how far ($\theta(\xi^k) > 0$) the solution is from satisfying the constraints. Take $p_k := g(\xi^k)$.

\subsubsection{SOS Projection} 
The projection onto the cone of sum-of-squares can be computed as the convex SOS program \begin{align} 
    \theta(\xi^k) := \min_{s \in \mathbb R[x]^{n_g}} \left\| {s}(x) -{p_k}(x) \right\|_2^2 \quad \text{s.t. $s \in \Sigma[x]^{n_g}$}
\end{align} 
where the optimal solution for $s$ is the nearest sum-of-squares polynomial to $p_k$.


\subsubsection{Signed Distance}
In~\cite[Section~4.3]{Cunis2025}, a signed-distance metric for generalized cones is introduced. It utilizes a scalarization function \cite{khan2015} to map $p_k$ to a finite \textit{signed distance} between $p$ and the boundary of $\Sigma[x]^{n_g}$. The signed-distance metric is defined as
\begin{equation}
	\mu_{\Sigma,s} : p \mapsto \inf\{r \in \mathbb R \mid p + r \odot s \in \Sigma[x]\}
\end{equation}
where the optimal value for $r$ is the scalar signed distance to the boundary of the cone, $\odot$ denotes the Hadamard product, and $s(x) \in \interior \Sigma[x]$. The SOS optimization problem 
	\begin{subequations}
    \label{eq:sgnddistance}
		\begin{align}
			  \min_{r \in\mathbb R^{n_g}} \quad & \quad \sum_{j=1}^{n_g} r_j  \quad
			 \text{s.t.}  \quad    p_k(x) + r \odot s(x) \in \Sigma[x]^{n_g}
		\end{align}
	\end{subequations}
is used to compute the signed-distance metric.
We define the overall distance, and hence constraint violation, as
\begin{equation}
    \theta(\xi^k) :=  \begin{cases}
        0 & \text{if } \mathrm{max}(r^*) \leq \epsilon \\
        \mathrm{max}(r^*) & \text{otherwise}.
    \end{cases}
\end{equation}
where $r^*$ is the optimal solution to \eqref{eq:sgnddistance} and $\epsilon > 0$ is a tolerance.

\subsubsection{Pseudo-Projection via Sampling}
Verifying that \(p_k(x) \in \Sigma[x]^{n_g} \) can be replaced by checking the necessary condition \( p_k(x) \geq 0  \)  on a finite set of sample points \( x^{(i)} \in \mathbb{R}^n \), \( i = 1, \dots, M \), and nonnegativity is checked element-wise. Since sampling over \( \mathbb{R}^n \) is computationally infeasible, the samples are drawn from a bounded hypercube.

To assess feasibility, the minimum value of each polynomial over the sampled points is computed as
\begin{equation}
    p_{\min} = \min_{\substack{j = 1,\dots, n_g \\ i = 1, \dots, M}} p_{k}(x^{(i)})_j 
\end{equation}
and the constraint violation is set to
\begin{equation}
    \theta(\xi^k) :=  \begin{cases}
        0 & \text{if } p_{\min} \geq 0 \\
        |p_{\min}| & \text{otherwise}.
    \end{cases}
\end{equation}
It should be noted that $\theta(\xi^k) = 0$ is neither necessary nor sufficient for $p_k \in \Sigma[x]^{n_g}$.


\subsubsection{Discussion}
The SOS projection and signed distance methods both define a distance metric for constraint violation via SDPs. The sampling approach is simple to implement and does not require an SDP solver, however, it is incomplete and its accuracy depends on the number of sample points.

Figure~\ref{fig:compConstraintVioalation} compares the computation times of these methods for increasing numbers $n$ of indeterminate variables, using 10 randomly generated SOS polynomials of degree six. The implementation is provided in the supplementary material~\cite{DARUS-5677_2026}. The projection method scales poorly, becoming impractical for larger $n$. 
In contrast, the signed distance remains efficient, while sampling with 1,000 points is fastest but least accurate.
Based on this comparison, we \added[id=TC]{implement} the signed distance method for its strong balance of accuracy and efficiency.
\begin{figure}[h!]
    \centering
    \setlength{\figH}{4cm}
     \setlength{\figW}{10cm}
%
%
\definecolor{mycolor1}{rgb}{1.00000,0.00000,1.00000}%
\definecolor{mycolor2}{rgb}{0.00000,1.00000,1.00000}%
\begin{tikzpicture}

\begin{axis}[%
width=0.951\figW,
height=\figH,
at={(0\figW,0\figH)},
scale only axis,
xmin=2,
xmax=5,
xtick={2,3,4,5}, 
xlabel style={font=\color{white!15!black}},
xlabel={$n$},
ymin=1e-2,
ymax=1e2,
ymode=log,
log basis y=10,
yminorticks=true,
minor y tick num=9, 
tick style={black, thin}, 
ylabel style={font=\color{white!15!black}},
ylabel={Computation Time (s)},
axis background/.style={fill=white},
xmajorgrids,
ymajorgrids,
yminorgrids,
legend style={at={(0.5,1.05)}, anchor=south, legend cell align=left, align=left, draw=white!15!black, column sep=1ex, legend columns=2} 
]
\addplot [color=red, line width=1.0pt, mark=o, mark options={solid, red}]
  table[row sep=crcr]{%
2	0.0503952\\
3	0.622989799999999\\
4	5.7394452\\
5	41.1370158\\
};
\addlegendentry{SOS Projection}

\addplot [color=green, line width=1.0pt, mark=asterisk, mark options={solid, green}]
  table[row sep=crcr]{%
2	0.0343625\\
3	0.044714\\
4	0.1323335\\
5	0.2840234\\
};
\addlegendentry{Signed Distance}

\addplot [color=blue, line width=1.0pt, mark=+, mark options={solid, blue}]
  table[row sep=crcr]{%
2	0.0179108\\
3	0.0146787\\
4	0.0239017\\
5	0.0410121\\
};
\addlegendentry{Sampling (1000)}

\addplot [color=mycolor1, line width=1.0pt, mark=x, mark options={solid, mycolor1}]
  table[row sep=crcr]{%
2	0.0201938\\
3	0.0538636\\
4	0.1274919\\
5	0.2870519\\
};
\addlegendentry{Sampling (10000)}

\addplot [color=mycolor2, line width=1.0pt, mark=square, mark options={solid, mycolor2}]
  table[row sep=crcr]{%
2	0.2069982\\
3	0.5309941\\
4	1.1561988\\
5	2.5298287\\
};
\addlegendentry{Sampling (100000)}

\end{axis}
\end{tikzpicture}%
    \caption{Performance comparison of methods used to measure constraint violation in SOS programming.}
    \label{fig:compConstraintVioalation}
\end{figure}
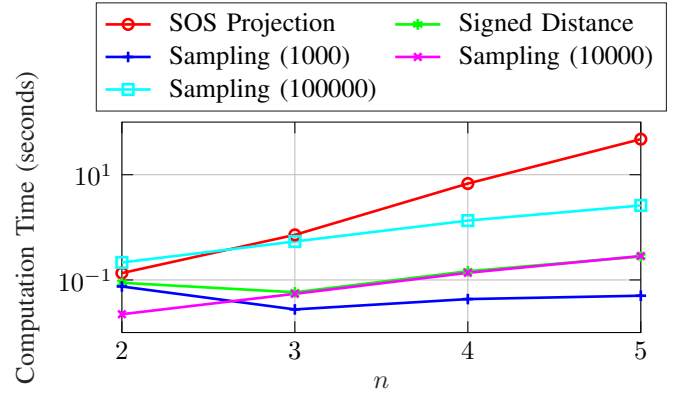

\subsection{Feasibility Restoration}
\label{subsec: FeasRes}
Small step lengths can cause the solver to stall or fail. If $\alpha^k < \alpha_{\text{min}}$ or infeasibility is detected in~\eqref{eq:QsubProb}, Alg.~\ref{alg: fullSQP} enters the feasibility restoration phase (lines 10 and 21 in Alg.~\ref{alg: fullSQP}) to maintain robustness~\cite[p.439]{nocedalNumericalOptimization2006}. This phase aims to reduce the constraint violation and provide a solution that is acceptable to the filter of the original problem.
\begin{remark}
    There is no guarantee that the restoration finds asolution acceptable to the filter. If it does not, we want it to converge to a local (infeasible) point. This information might be used by the user to take corrective action on the problem formulation~(cf.~\cite{wachterImplementationInteriorpointFilter2006a}).
\end{remark}

The feasibility restoration might be invoked after certain progress has \added[id=TC]{already} been made towards a local solution, but the constraint violation is still too high. In that case, we want the restoration not to deviate \added[id=TC]{too} much from the current, potentially close to (local) optimal solution. To achieve this, we introduce a weighted regularization term based on the current constraint violation~\cite{wachterImplementationInteriorpointFilter2006a,ulbrich2004b}. 

The high-level SOS feasibility restoration problem with regularization reads
\begin{subequations}
\label{eq: FeasResSoSProb}
	\begin{align}
		\min_{r \in \mathbb{R}^{n_g}, \xi_R \in \mathbb{R}^n} & \quad \sum_{i=1}^{n_g} r_i + \frac{\rho_k}{2} \|\xi_R - \xi^k\|_2^2  \\
		\text{s.t.} & \quad g(\xi_R) + r \odot s \in \Sigma[x]^{n_g}
	\end{align}
\end{subequations}
where $\rho_k > 0$ is a fixed penalty parameter, and $\xi^k$ is the iterate at which the restoration phase is initiated. Problem~\eqref{eq: FeasResSoSProb} is similar to~\eqref{eq:sgnddistance}, but it differs in that both the decision variables and the constraint violations are optimized simultaneously. The additional decision variables introduced are negligible compared to the overall problem size. 
Due to the explicit reduction of constraint violation in the cost function, previous approaches for nonlinear parameter optimization, such as~\cite{wachterImplementationInteriorpointFilter2006a}, needed a smooth reformulation of the problem using slack variables. In contrast, our approach {is already smooth}.

The penalty parameter is heuristically determined based on the constraint violation (signed-distance) of the current iterate. In our implementation, we make use of the following power function
\begin{equation}
    \rho_k = \begin{cases}
        \rho_\text{max} & \text{if $\theta(\xi^k) \leq \epsilon_{\text{feas}}$} \\
        \epsilon_{\text{feas}} / \theta(\xi^k) + \rho_\text{min} & \text{else} 
    \end{cases}
\end{equation}
with $ 0 <  \rho_{\text{min}} < \rho_{\text{max}}$ defined by the user. In case of large constraint violation, more weight is put onto feasibility, and vice-versa. The penalty $\rho_k$ is {computed} once at the {initiation} of the feasibility restoration.

Problem~\eqref{eq: FeasResSoSProb} is again nonlinear and hence we apply the previously described filter line search algorithm to solve this problem. Two modifications are made: 
\begin{enumerate}
    \item The restoration phase does not have its own restoration phase and, 
    \item the restoration phase is left once an iterate is acceptable to the filter of the original problem with an additional threshold on constraint violation, {which takes the form} $\theta(\xi_R^k) \leq \eta \theta(\xi^k)$ with $\eta \in [0,1]$  (cf.~\cite{wachterImplementationInteriorpointFilter2006a}). 
\end{enumerate}
The adapted algorithm for the feasibility restoration phase is outlined in Alg.~\ref{alg: feasRes}.
In the current implementation, the restoration is not invoked if the constraint violation is already below the assigned threshold.

\begin{algorithm}[h!]
\caption{Feasibility Restoration Phase}\label{alg: feasRes}
\KwIn{Initial guess $\xi^0$, $s^0$, parameters}
Initialize Hessian $R_0$, Augmented filter, $k = 0$\\
\While{$k \leq  N_\mathrm{max} $}{
    \text{Solve}~\eqref{eq:QsubProb} \text{to obtain} $\xi_R^{*}$ \text{and} $\lambda_R^{*}$  and set $\alpha_k = 1$\\
    \eIf{Feasible}
    {
    Execute filter line search (Algorithm \ref{alg: filterLineSearch})\\
    \If{$\alpha \leq \alpha_\mathrm{min}$}
    {
        Solver stalled-Terminate
    }

    }
    {
        Problem seems infeasible-Terminate
    }

    \eIf{ Acceptable to augmented filter and
         $|| \theta(\xi^{k+1})|| \leq \kappa \theta(\xi^{k})$ }
    {
        Terminate\\
    }
    {
        Compute  Hessian approximation\\
        set $k = k+1$ 
    }
 
}
\end{algorithm}

\subsection{Implementation Details}
This section provides additional implementation details and practical aspects. This includes Hessian approximation, the details about the implemented second-order correction and parameters used. 
\subsubsection{Hessian Approximation}
\label{subsubsec: HessianApprox}
In order to obtain an accurate but also positive definite estimate of the Hessian, several methods can be found in the literature. Our implementation provides several options \added[id=TC]{of which the} applicability depends, e.g., \added[id=TC]{on the specific} problem size.  \added[id=TC]{Options include} the damped Broyden-Fletcher-Goldfarb-Shanno (BFGS) update~\cite[Procedure 18.2]{nocedalNumericalOptimization2006}  and exact Hessian computation in combination with regularization techniques \added[id=TC]{such as} the Gershgorin bound~\cite[Section~2.5]{betts2010}, minimum Frobenius norm approach~\cite[p.~50]{nocedalNumericalOptimization2006}  or using the \textit{mirrored} version of the exact Hessian~\cite{quirynen2014}. 
The initialization of the Hessian is also crucial for convergence of the algorithm. The matrix $H_0$ in Alg.~\ref{alg: fullSQP} \added[id=TC]{is} initialized with the identity matrix, \added[id=TC]{unless} the user provides a scaling factor to improve convergence (cf.~\cite[Chapter~6]{nocedalNumericalOptimization2006}) or \replaced[id=TC]{the exact (regularized) Hessian is used}{makes use of exact Hessian computation in combination with regularization}.

\subsubsection{Second-Order Correction}
\label{subsec:SOC}
Merit functions or filters may converge slowly because large steps are often rejected, a phenomenon known as the {\em Maratos effect}. A potential cause of slow convergence is an inaccurate linear approximation of the constraint function. To mitigate this, \added[id=TC]{a second-order correction (SOC)} step \added[id=TC]{is} applied to adjust the search direction and \added[id=TC]{further} reduce constraint violations \cite{nocedalNumericalOptimization2006}. Quadratic approximations of the constraint functions would address this but lead to a challenging quadratically-constrained quadratic program. An alternative, outlined in \cite[p.~544]{nocedalNumericalOptimization2006}, is employed here. The approach involves replacing the quadratic term in the constraint function’s quadratic approximation with a constant quadratic term based on the current uncorrected search direction $\omega^k = (\xi_+ - \xi^k)$ (see \cite[p.~544]{nocedalNumericalOptimization2006} for details). \added[id=TC]{Specifically}, the constant term 
\begin{align*}
    g(\xi^k + \omega^k) - \nabla g(\xi^k)^\top \omega^k 
\end{align*}
is added to \eqref{eq: linConQsub}
\added[id=TC]{and the} adapted quadratic subproblem is \added[id=TC]{then} solved to obtain $\xi_\text{soc}^*$ which is used to adjust the search direction. 
The \added[id=TC]{thus corrected,} new trial point is then given by (with $\alpha = 1$)
\begin{equation}
    \label{eq: correctedSOCStep_trialPoint}
    \xi^{k+1} = \xi_\text{soc}^* + \omega^k
\end{equation}
The new \added[id=TC]{trial point} from \eqref{eq: correctedSOCStep_trialPoint} is also checked for filter acceptance. If it is not possible to find a feasible point, the SOC returns unsuccessful and step-length is reduced.

\subsubsection{Scaled Termination and Parameter}
\label{subsec: PractAsp}
To improve robustness against numerical errors and poor scaling, we make use of a scaled KKT condition \added[id=TC]{given by}~\cite{rauVski}
\begin{equation}
\left\| \nabla_\xi L(\xi^k, \lambda^k) \right\|_\infty \leq 
\frac{\epsilon_{\text{opt}} \, \max \left( 1, |f(\xi^k)| \right) + 
\left\| \langle \lambda^k , g(\xi^k) \rangle \right\|_\infty}
{\left\| \omega^k \right\|_\infty}
\end{equation}
that ensures
\begin{equation}
|f(\xi^{k+1}) - f(\xi^k)| \leq \epsilon_{\text{opt}}.
\label{eq: scaledOptCon}
\end{equation}
For a derivation of the scaled KKT condition see~\cite{nikolayzik2012}. 
In our implementation, we set the tolerances and parameters of the filter as follows: $\epsilon_{\text{opt}} = 10^{-4}$, $\epsilon_{\text{feas}} = 10^{-6}$, $s_f = 2$, $s_\theta = 0.9$, $\delta = 1$ and $\rho = 10^{-4}$, $\gamma_f = \gamma_\theta = 10^{-5}$, $\eta = 10^{-4}$,  
{$\rho_\mathrm{min} = 0.01$ and $\rho_\mathrm{max} = 1$}.

\section{Benchmarks}
\label{sec: Benchmarks}
We provide several numerical benchmarks \added[id=TC]{arising} from nonlinear system analysis and control design. Only nonconvex SOS problems with potentially nonaffine system dynamics are considered. The proposed filter line search algorithm  is benchmarked \added[id=TC]{against} the \added[id=TC]{hitherto} state-of-the-art method, coordinate-descent (CD, cf. ~\cite{chakrabortyEtAl2011,chakrabortyEtAl2011ii,iannelli2018,seilerBalas2010,yinEtAl2019,Yin2021,jarvisEtAl2003,Cunis2021aut,majumdarEtAl2013,lin2022,newton2025a,amesControlBarrierFunctions2019,clark2021,tan2006,schneeberger_sos_2023,loureiro2025}).
This section concludes with a comparison of the main characteristics of the two methods and \added[id=TC]{some} practical aspects such as initial guess generation for the sequential \added[id=TC]{quadratic SOS} approach. 

\subsection{Preliminaries}
All benchmarks are implemented in CaΣoS~\cite{Cunis2025} v1.0.0 and are available in the supplementary material~\cite{DARUS-5677_2026}. In CaΣoS, the SOS problems are parameterized once at the beginning, which we refer as build process with its corresponding build time. The solve time is the total time spent in the iterative procedure. We make use of MOSEK \cite{mosek} v11.1.8 as the underlying SDP solver. 
\added[id=TC]{As we are not aware of existing} benchmarks for nonconvex SOS problems, we make use of several SOS problems arising in systems and control \added[id=TC]{theory} with some modifications. Hence, detailed derivations of the SOS problem are omitted. The system dynamics and the considered polynomials are provided in the supplementary material~\cite{DARUS-5677_2026}. If not differently stated, we assume CD is converged if $|f_k - f_{k-1}| \leq \epsilon_{\text{opt}}$, which is equivalent to the scaled termination criterion in~\eqref{eq: scaledOptCon}. The maximum number of iterations is set to $N_{\text{max}} =  100$. We compare the number of iterations, total computation  time (total solve time + solver build time) and final cost for a given problem size. We denote the problem size by the number of system states $n$ \added[id=TC]{and the number of} constraints $n_\text{con.}$ (both linear and conic constraints) and decision variables $n_\text{dec.}$ of the underlying SDP. We report only the problem size of the sequential \added[id=TC]{quadratic} SOS approach, as it reflects the overall size. In contrast, CD involves multiple smaller subproblems, which may include redundant constraints or variables. A detailed breakdown is available in the supplementary material~\cite{DARUS-5677_2026}.

All results were computed on a \added[id=TC]{personal} computer with Windows 11 and MATLAB 2023b running on an AMD Ryzen 9 5950X 16-Core Processor with 3.40 GHz and 128 GB RAM.

\subsection{Region-of-Attraction Estimation}
We consider the problem \added[id=TC]{of computing an inner approximation for} the region-of-attraction (ROA) of nonlinear closed-loop systems.
We solve the standard ROA estimation problem
\begin{subequations}
\label{eq: SOSprob}
\begin{alignat}{2}
\min_{\substack{
    V\in \mathbb{R}[x], s \in \Sigma[x]
}} 
& \quad \|g_0 - V\|^2_{\mathbb R[x]} \label{eq: obj} \\[1ex]
\text{s.t.} \quad
& V - \epsilon(x) 
&& \in \Sigma[x] \label{eq: lyap1} \\
& s (V - \gamma) - \langle \nabla V ,f\rangle - \epsilon(x)  
&& \in \Sigma[x] \label{eq: lyap2} 
\end{alignat}
\end{subequations}
where $V$ is the sought Lyapunov function, $s$ is a SOS multiplier resulting from Theorem~\ref{Theo: genSprocedure} and $f(\cdot)$ is the closed loop dynamic, $g_0$ is a pre-defined \textit{safe set} which we make use of to find a less conservative solution  by minimizing the squared $l_2$-norm $\|\cdot\|_{\mathbb R[x]}$ on the space of polynomials. A small cost value indicates large portions of the safe set are covered by the ROA function. Thus, such a cost function is a good choice for (state) constrained problems, to obtain less conservative results and is also used in subsequent benchmarks in a similar fashion.
The stable level set is \added[id=TC]{fixed to} $\gamma = 1$, but could be also a decision variable.
Two examples are considered: The first example is taken from \cite{chakrabortyEtAl2011ii}, where the nonlinear ROA analysis is used to assess the robustness properties of the F/A-18 flight control laws. 
The second example evaluates scalability and robustness in terms of the number of system states $n$, using an $N$-link (planar) robot arm.

\subsubsection{F/A-18 Flight Control Laws}
We take the seven state closed-loop dynamics from~\cite[Appendix~B]{chakrabortyEtAl2011ii}. For both methods, we compute an initial Lyapunov candidate by linearizing the nonlinear closed-loop dynamics. For the sequential \added[id=TC]{quadratic SOS} approach, we additionally \added[id=TC]{choose the} initial SOS multiplier $s= \sum_i x_i^2$. Due to infeasibility of CD at the first iteration, we solved the first subproblem via bisection. After the first iteration, we used the pre-defined level set. The results of the two approaches are summarized in Table~\ref{tab:compROA_F18}. 

The sequential \added[id=TC]{quadratic SOS} approach only needs 19 iterations until convergence, whereas coordinate descent did not converge within the pre-assigned maximum number of iterations. The total time (solver build + solve time) as stated in Table~\ref{tab:compROA_F18} is substantially less for the sequential \added[id=TC]{quadratic SOS} approach.
The build time for the sequential \added[id=TC]{quadratic SOS} approach is \SI{4.81}{\second} \added[id=TC]{opposed to only \SI{1.15}{\second}} for coordinate descent, which includes several subproblems.

\begin{table}[h!]
  \caption{Comparison of sequential \added[id=TC]{quadratic} SOS (\added[id=TC]{SQ}) and coordinate descent (\added[id=TC]{CD}) for the F/A-18 ROA estimation problem.}
  \label{tab:compROA_F18}
  \centering
  \renewcommand{\arraystretch}{1.2} 
  \setlength{\tabcolsep}{6pt} 
  \begin{tabular}{cc cc cc cc c} 
    \hline
  \hline
    \textbf{n} & \multicolumn{2}{c}{\textbf{Size}} & \multicolumn{2}{c}{\textbf{Iter. [-]}} & \multicolumn{2}{c}{\textbf{Total time [s]}} & \multicolumn{2}{c}{\textbf{Cost [-]}} \\ 
               & {n$_\text{con.}$} & {n$_\text{dec.}$} & SQ & CD & SQ & CD & SQ & CD \\ 
    \hline
    7  & 2058 &15785 &  19 & 100 & 52.5 &   277.27    & 2.42  &   4557.3 \\
      \hline
  \hline
  \end{tabular}
\end{table}

\subsubsection{N-link Robot Arm}
Using an $N$-link (planar) robot arm, the problem scales to $n=2N$ states.  For each $N$, the dynamics and a linear quadratic control law (LQR) are computed and a quadratic polynomial approximation for the closed-loop system is calculated. This preparation phase is provided in the supplementary material~\cite{DARUS-5677_2026}.
We solve the problem twice using different initial guesses. In the first case, both algorithms are initialized with the Lyapunov function of the linear system, with the sequential \added[id=TC]{quadratic SOS} algorithm additionally \added[id=TC]{given} a simple quadratic polynomial for the SOS multiplier. In the second, we make use of a clearly infeasible initial guess. A more detailed breakdown is given in Table~\ref{tab:compNlink} and ~\ref{tab:compNlinkbadinit}.

\begin{table}[h!]
  \caption{Comparison of sequential \added[id=TC]{quadratic} SOS (SQ) and coordinate descent (CD) for the N-link robotic arm ROA estimation problem using a good initial guess. \added[id=TC]{$^\ddagger$ indicates that CD was initially infeasible.}}
  \label{tab:compNlink}
  \centering
  \renewcommand{\arraystretch}{1.2}
  \setlength{\tabcolsep}{6pt}
  \sisetup{
    table-alignment-mode = format
  }
  \begin{tabular}{@{}
    S[table-format=2.0, table-number-alignment=center] 
    S[table-format=6.0, table-number-alignment=right]
    S[table-format=6.0, table-number-alignment=right] 
    S[table-format=3.0, table-number-alignment=right]
    S[table-format=3.0, table-number-alignment=right] 
    S[table-format=6.1, table-number-alignment=right]
    S[table-format=6.1, table-number-alignment=right]
  @{}}
  \hline
  \hline
    \textbf{n} &
    \multicolumn{2}{c}{\textbf{Size}} &
    \multicolumn{2}{c}{\textbf{Iterations [-]}} &
    \multicolumn{2}{c}{\textbf{Time [s]}} \\
    &
    {{n$_\text{con.}$}} & {n$_\text{dec.}$} &
    {\textbf{SQ}} & {\textbf{CD}} &
    {\textbf{SQ}} & {\textbf{CD}} \\
    \hline
    4  &  85   &   248   & 7      & 23              &      0.2     &    1.3      \\
    6  &  245   &  843   & 9      & 15                  &      0.8     &   1.0  \\
    8   & 558   &  2136  & 9      & 17              &      2.2      &   3.1       \\
    10  & 1100 & 4535    & 10     & 16              &     8.4       &   9.5       \\
    12  & 1963 & 8544    & 11     & 23              &      29.3     &   42.4      \\
    14  & 3255 & 14763   & 10     & 26              &      74.6     &   153.3     \\
    16  & 5100 & 23888   & 9      & 29$^\ddagger$     &      179.9   &   547.3    \\
    18  & 7638 & 36711   & 10     & 33$^\ddagger$    &      504.8    &   1477.5   \\
    20 &  11025 &54120   & 11     & 28$^\ddagger$     &   1032.4     &   2735.9   \\
    22 & 15433 & 77099   &  10    & 32$^\ddagger$     &   2318.1     &   13106.1  \\
    24 & 21050 & 106728  &  11    & 50$^\ddagger$    &   5811.7     &   36668.1   \\
      \hline
  \hline
  \end{tabular}
\end{table}

\paragraph{Good Initial Guess}
The results are summarized in Table~\ref{tab:compNlink}. Both methods achieve optimality (using the assigned thresholds). Final costs are omitted because they are nearly identical, with the sequential \added[id=TC]{quadratic SOS approach} performing slightly better. Notably, the sequential \added[id=TC]{quadratic SOS approach} requires significantly less iterations than coordinate descent—especially for larger number of states—resulting in lower total computation time as shown in~Fig.~\ref{fig:nlink_comparsion} and Table~\ref{tab:compNlink}. In some cases, \added[id=TC]{indicated by $^\ddagger$,} CD was infeasible in the first subproblem \added[id=TC]{of the first iteration}. In these cases we gave CD a \textit{grace} iteration and tried to solve the second subproblem anyway. On average, the sequential \added[id=TC]{quadratic SOS approach} requires \SI{60}{\percent} less iterations and \SI{54}{\percent} less time to solve the problems. 

\begin{table}[h!]
  \caption{Comparison of sequential \added[id=TC]{quadratic} SOS (SQ) and coordinate descent (CD) for the N-link pendulum ROA estimation problem using a bad initial guess. \added[id=TC]{$^*$ indicates convergence to a feasible solution.}}
  \label{tab:compNlinkbadinit}
  \centering
  \renewcommand{\arraystretch}{1.2}
  \setlength{\tabcolsep}{6pt}
  \sisetup{
    table-alignment-mode = format
  }
  \begin{tabular}{@{}
    S[table-format=2.0, table-number-alignment=center] 
    S[table-format=6.0, table-number-alignment=right]
    S[table-format=6.0, table-number-alignment=right] 
    S[table-format=3.0, table-number-alignment=right]
    S[table-format=3.0, table-number-alignment=right] 
    S[table-format=6.1, table-number-alignment=right]
    S[table-format=6.1, table-number-alignment=right]
  @{}}
    \hline
  \hline
    \textbf{n} &
    \multicolumn{2}{c}{\textbf{Size}} &
    \multicolumn{2}{c}{\textbf{Iter. [-]}} &
    \multicolumn{2}{c}{\textbf{Time [s]}} \\
    &
    {{n$_\text{con.}$}} & {n$_\text{dec.}$} &
    {\textbf{SQ}} & {\textbf{CD}} &
    {\textbf{SQ}} & {\textbf{CD}} \\
    \hline
    4       &  85   &   248   & 45  & {--}    &      3.6    &   {--}       \\
    6       &  245   &  843   & 56  & {--}    &      7.8   &   {--}        \\
    8       & 558   &  2136  & 73  & {--}    &      30.9  &    {--}       \\
    10      & 1100 & 4535    & 11   & {--}    &     16.7   &   {--}       \\
    12      & 1963 & 8544    & 12  & {--}    &      52.3   &  {--}        \\
    14  & 3255 & 14763   & 11$^*$  & {--}    &      150.2   &  {--}   \\
    16  & 5100 & 23888   & 12$^*$ & {--}      &     454.2  & {--}     \\
    18      & 7638 & 36711   & 16   & {--}    &     1354.01   & {--}      \\
    20 &  11025 &54120   & 15$^*$  & {--}        &  3073.6    & {--}  \\
    22 & 15433 & 77099   & 14$^*$    &   {--}    &  9873.7   &    {--}\\
    24      & 21050 & 106728  & 15    &  {--}     &  22245.9    &    {--} \\
      \hline
  \hline
  \end{tabular}
\end{table}

\paragraph{Bad Initial Guess}
The second case considers the case with a \textit{bad} initial guess. We initialize all decision variables with negative quadratic polynomials. Such polynomials are clearly infeasible. 
Coordinate descent is not able to start the computation because it needs a feasible initial guess. Thus, no results can be reported in this case. In contrast, the proposed sequential algorithm with its feasibility restoration is able to solve the problem to full optimality or, in a few cases, at least to a feasible solution\footnote{\added[id=TC]{That is,} the thresholds for constraint violations are met, but the solution might be far from optimal.} (indicated by $^*$ in Table~\ref{tab:compNlinkbadinit}). This clearly shows the enhanced robustness of the proposed approach. The call of the feasibility restoration phase increases the number of iterations significantly for smaller number of states. For medium to large number of states, the number of iterations is only moderately increased. However, the average time per iteration is much larger compared to cases where the feasibility restoration is not called, which is expected.

\begin{figure}[h!]
    \centering
    \setlength{\figH}{6.5cm}
     \setlength{\figW}{7.5cm}
%
%
\definecolor{mycolor1}{rgb}{0.06600,0.44300,0.74500}%
\definecolor{mycolor2}{rgb}{0.12941,0.12941,0.12941}%
\begin{tikzpicture}

\begin{axis}[%
width=0.951\figW,
height=0.419\figH,
at={(0\figW,0.581\figH)},
scale only axis,
bar shift auto,
xmin=1.6,
xmax=26.4,
ymin=0,
ymax=90,
xtick={ 4,  6,  8, 10, 12, 14, 16, 18, 20, 22, 24},
xlabel={Number of states $n$},
ylabel style={font=\color{mycolor2}},
ylabel={$\frac{T_\mathrm{CD}-T_\mathrm{SQ}}{T_\mathrm{CD}}$ in \%},
axis background/.style={fill=white},
axis x line*=bottom,
axis y line*=left,
xmajorgrids,
ymajorgrids,
legend columns = 2,
legend style={at={(0.95,1.2)},legend cell align=left, align=left},
]
\addplot[ybar, bar width=1.6, fill=mycolor1, draw=mycolor2, area legend] table[row sep=crcr] {%
4	80.4266662539802\\
6	26.6406484608838\\
8	28.7440375148791\\
10	11.2985033977873\\
12	30.8188430073692\\
14	51.321440634858\\
16	67.1412966970301\\
18	65.8353516994445\\
20	62.265147626029\\
22	82.3135726727546\\
24	84.1506361664627\\
};
\addlegendentry{Rel. Improvement}
\addplot[mark=none, black,dashed, domain=0:50] {53.7233};
\addlegendentry{Mean value}

\end{axis}

\begin{axis}[%
width=0.951\figW,
height=0.419\figH,
at={(0\figW,0\figH)},
scale only axis,
bar shift auto,
xmin=1.6,
xmax=26.4,
xtick={ 4,  6,  8, 10, 12, 14, 16, 18, 20, 22, 24},
xlabel style={font=\color{mycolor2}},
xlabel={Number of states $n$},
ymin=0,
ymax=80,
ylabel style={font=\color{mycolor2}},
ylabel={$\frac{I_\mathrm{CD}-I_\mathrm{SQ}}{I_\mathrm{CD}}$ in \%},
axis background/.style={fill=white},
axis x line*=bottom,
axis y line*=left,
xmajorgrids,
ymajorgrids,
legend style={legend cell align=left, align=left}
]
\addplot[ybar, bar width=1.6, fill=mycolor1, draw=mycolor2, area legend] table[row sep=crcr] {%
4	69.5652173913043\\
6	40\\
8	47.0588235294118\\
10	37.5\\
12	52.1739130434783\\
14	61.5384615384615\\
16	68.9655172413793\\
18	69.6969696969697\\
20	60.7142857142857\\
22	68.75\\
24	78\\
};

\addplot[mark=none, black, dashed,domain=0:50] {59.45};

\end{axis}
\end{tikzpicture}%
    \caption{Relative performance improvement of sequential \added[id=TC]{quadratic} SOS (SQ) with respect to coordinate-descent (CD) in terms of total computation time $T_{(\cdot)}$ and \added[id=TC]{number of} iterations $I_{(\cdot)}$.}
    \label{fig:nlink_comparsion}
\end{figure}
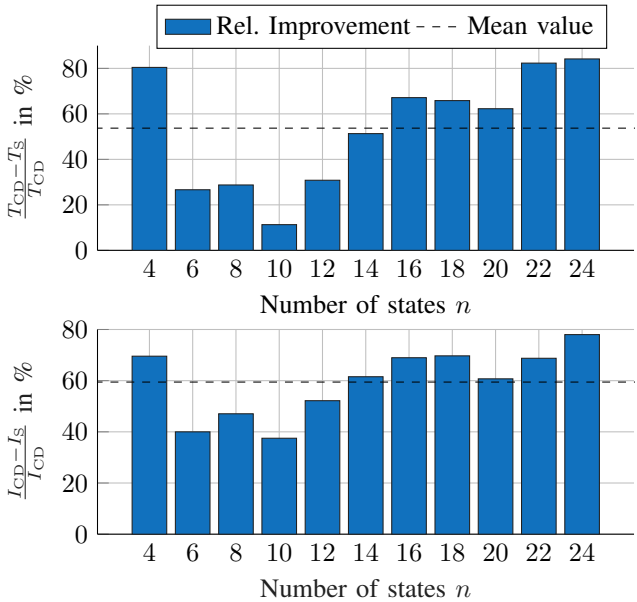

\begin{remark}
    It is unlikely that someone uses such an obviously infeasible initial guess and rather uses an initial guess similar to the first case. This is only made for demonstration purpose. It should be noted that an initial guess like in the first case does not guarantee convergence for coordinate descent as demonstrated in the previous \deleted[id=TC]{F/A-18} example.
\end{remark}

\subsection{Control Synthesis for Nonlinear Systems}
\label{subsec: caseII}
\added[id=TC]{We} consider the synthesis of a stabilizing control law for nonlinear systems \added[id=TC]{subject to} state and control constraints. Such problems are \added[id=TC]{also} relevant \added[id=TC]{for computing} terminal conditions \added[id=TC]{that} ensure asymptotic stability of nonlinear model predictive control.
We consider the input constraint set as polyhedral set
\begin{align*}
    \Uset = \{ u \in \mathbb R^m ~|~ H_\Uset u \leq \mathbf 1_p\}
\end{align*}
with $H_\Uset \in \mathbb R^{p \times m}$, $p \in \mathbb N$, and $\mathbf 1_p = (1,..,1)\in R^p$.
We solve the \added[id=TC]{nonconvex} SOS problem
\begin{subequations}
\label{eq: SOSprob1}
\begin{alignat}{2}
\min_{\substack{
    V, \hat\kappa \in \mathbb{R}[x], \\
    s_1, s_2 \in \Sigma[x], \\
    s_3 \in \Sigma[x]^m
}} 
& \quad \|h - V\|^2_{\mathbb R[x]} \label{eq: obj} \\[1ex]
\text{s.t.} \quad
& V - \epsilon(x)  
&& \in \Sigma[x] \label{eq: lyap1} \\
& s_1 (V - \beta) - \langle \nabla V,f(x, \hat\kappa)\rangle - \epsilon(x) 
&& \in \Sigma[x] \label{eq: lyap2} \\
& s_2 (V - \beta) - h(x) 
&& \in \Sigma[x] \label{eq: stateCon} \\
& s_3 (\hat h-\beta) - (H_\mathcal{U} \hat\kappa - \mathbf 1_p)  
&& \in \Sigma[x]^m \label{eq: lowerContControlDess}
\end{alignat}
\end{subequations}
where $V(\cdot)$ is a the sought Lyapunov function, $\beta > 0$ is the level set (either fixed or an additional decision variable), $\hat\kappa$ is the sought polynomial control law, and $s(\cdot)$ are the SOS multipliers resulting from Theorem~\ref{Theo: genSprocedure}. Eqs.~\eqref{eq: lyap1} and \eqref{eq: lyap2} are standard Lyapunov conditions \added[id=TC]{whereas} \eqref{eq: stateCon} and \eqref{eq: lowerContControlDess} enforce state and control constraints, respectively. We make use of a similar cost function as before. 
 
Again two examples are considered: The first example considers \added[id=TC]{the} nonaffine nonlinear dynamics of an aircraft's longitudinal motion with four states and two control \added[id=TC]{inputs}, which cannot be solved with coordinate descent in a direct manner. The second showcases the synthesis of a control law for control-affine satellite attitude dynamics. 
  
\subsubsection{Nonlinear Aircraft Dynamics}
In this example, we make use of the four-state longitudinal motion of an aircraft provided in~\cite{chakrabortyEtAl2011}. Compared to the previous examples, the control \added[id=TC]{inputs} enter nonaffine\added[id=TC]{ly} into the dynamics and hence, \added[id=TC]{the synthesis problem} cannot be solved by coordinate descent in a direct manner. For synthesis, we must re-write the system dynamics into a \added[id=TC]{control-}affine form. This would be possible by augmenting the system dynamics with a zero dynamic of the controls, resulting in $n_e =  n + m$ states and hence in a higher computational load. In contrast, with the sequential \added[id=TC]{quadratic SOS} approach we are able to directly solve this problem. In Table~\ref{tab:compConROA_GTM_State}, we provide only results for the sequential \added[id=TC]{quadratic} SOS approach.  In about \SI{13}{\second} one obtains a \textit{safe-by-design} \added[id=TC]{feedback law} for the nonlinear system.

\subsubsection{Spacecraft Attitude Control}
\label{subsubsec: caseII_sat}
The second example considers the control synthesis for a Hubble\added[id=TC]{-like spaceborne} telescope~\cite{nurre1995} along with attitude, rate and control constraints. The details can be found in the supplementary material~\cite{DARUS-5677_2026}. The \added[id=TC]{six-state dynamics are comprised of} angular rates and Modified-Rodrigues parameter\added[id=TC]{s}, \added[id=TC]{the latter of} which describe the spacecraft attitude. The results are summarized in the second row of Table~\ref{tab:compConROA_GTM_State}. Whereas the sequential \added[id=TC]{quadratic SOS} approach solves the problem in slightly above one minute to (local) optimality, the coordinate descent approach failed after six iterations (indicated by $^\dagger$), far from optimality.

\begin{table}[h!]
  \caption{Comparison of sequential \added[id=TC]{quadratic} SOS (SQ) and coordinate descent (CD) for the constrained \added[id=TC]{control synthesis problems}. \added[id=TC]{$^\dagger$ indicates failure.}}
  \label{tab:compConROA_GTM_State}
  \centering
  \renewcommand{\arraystretch}{1.2} 
  \setlength{\tabcolsep}{6pt} 
  \begin{tabular}{cc cc cc cc c} 
    \hline
  \hline
    \textbf{n} & \multicolumn{2}{c}{\textbf{Size}} & \multicolumn{2}{c}{\textbf{Iter. [-]}} & \multicolumn{2}{c}{\textbf{Time [s]}} & \multicolumn{2}{c}{\textbf{Cost [-]}} \\ 
               & {n$_\text{con.}$} & {n$_\text{dec.}$} & SQ & CD & SQ & CD & SQ & CD \\ 
    \hline
    4  & 1099 &  7333 & 20  & {--} & 12.5  &   {--}  & 0.34   &  {--} \\
    6  & 5543 &  57048 & 5  & 6$^\dagger$ & 73.9  &   91.1  & 0.01   &  29.6 \\
      \hline
  \hline
  \end{tabular}
\end{table}

\subsection{Reachability Analysis}
\added[id=TC]{We consider} the computation of \added[id=TC]{an} inner approximation of the reachable set \added[id=TC]{for} two different dynamical systems with two and four states, respectively. The problem formulation and the derivation can be found in~\cite{Yin2021,Cunis2021aut}. 
A state constraint set of the form $\mathcal{X} = \{x \in \mathbb R^n \mid h(x) \leq 0\}$ is \added[id=TC]{prescribed} and the goal is to \added[id=TC]{find} a set of the form {$\Omega^V_{t,\beta} = \{x \in \mathbb R^\ell \mid V(x,t) \leq \beta \}$}, where $V:\mathbb R^\ell \times \mathbb R \rightarrow \mathbb R$ is a \textit{reachability storage function} and $\beta > 0$ bounds the set, together with a time-varying control law $u(t) =\hat\kappa(t,x(t))$. 
 
We slightly modify the problem by using a similar cost as in the first benchmark instead of \added[id=TC]{the} bisection approach in~\cite{Yin2021}. We solve the \added[id=TC]{nonconvex} SOS problem (cf.\cite{Yin2021,Cunis2021aut})
\begin{subequations}
\label{eq: SOSprob2}
\begin{alignat}{2}
\min_{\substack{
    s_0 \in \Sigma[x],\beta \in \mathbb R,\\
    V, \hat\kappa \in \mathbb{R}[t,x], \\
    s_1, s_2, s_3, s_4 \in \Sigma[x,t],\\
    s_5, s_6, \in \Sigma[x,t]^m\\
}} &
\quad \|h - V\|^2 \label{eq: cost2SOS} \\
\text{s.t.} \quad  &s_0 V(T,x) - P && \in \Sigma[x] \label{eq: termSetSOS} \\
&s_1 (V - \beta) - s_2 g - h && \in \Sigma[t,x] \label{eq: safeSetSoS} \\
&s_3 (V - \beta) - s_4 g - \tau && \in \Sigma[t,x] \label{eq: dissipSoS} \\
&s_5 (V - \beta) - s_6 g - (H_\mathcal{U} \hat\kappa - \mathbf 1_p)   && \in \Sigma[t,x]^m \label{eq: lowerContSOSRe}
\end{alignat}
\end{subequations}
with  $\tau = \nabla_t V + \langle \nabla_x V, f(x, \hat\kappa) \rangle$  where  $s_{(\cdot)}$ are the SOS multipliers resulting from Theorem~\ref{Theo: genSprocedure}, $P$ is a polynomial describing the terminal set, and the polynomial {$g = (t-t_0)(T-t)$} ensures continuous-time constraint satisfaction for $t \in [t_0, T]$.
Eq.~\eqref{eq: termSetSOS} ensures that the level set of the storage function lies in the terminal set at $t = T$, Eq.~\eqref{eq: safeSetSoS}  ensures state constraint  satisfaction, Eq.~\eqref{eq: dissipSoS} is a dissipation inequality that ensures convergence, and Eq.~\eqref{eq: lowerContSOSRe} ensures control constraint satisfaction. 

The results are summarized in Table~\ref{tab:reachabilityComp}. 
Both algorithms are initialized using the Lyapunov function from the linearized system. For the sequential \added[id=TC]{quadratic SOS approach}, all multipliers are initialized with simple quadratic polynomials.

\begin{table}[h!]
  \caption{Comparison of sequential \added[id=TC]{quadratic} SOS (SQ) and coordinate descent (CD) for reachability analysis for different system sizes.}
  \label{tab:reachabilityComp}
  \centering
  \renewcommand{\arraystretch}{1.2} 
  \setlength{\tabcolsep}{6pt} 
  \begin{tabular}{cc cc cc cc c} 
    \hline
  \hline
      \textbf{n} &\multicolumn{2}{c}{\textbf{Size}} & \multicolumn{2}{c}{\textbf{Iter. [-]}} & \multicolumn{2}{c}{\textbf{Time [s]}} & \multicolumn{2}{c}{\textbf{Cost [-]}} \\ 
               & {n$_\text{con.}$} & {n$_\text{dec.}$} & SQ & CD & SQ & CD & SQ & CD \\ 
    \hline
     2  & 943   & 6042 &  51 & 78 & 22.5 &   21.5  & 0.84  &  0.86 \\
      4  &  3171 & 27647  & 11  & 100 & 42.8  & 371.2    &  0.04 & 0.39\\
        \hline
  \hline
  \end{tabular}
\end{table}

In the \added[id=TC]{2D example}, the sequential \added[id=TC]{quadratic SOS approach} needs significantly less iterations, yet both approaches take \added[id=TC]{a similar} total time to solve the problem.
For the 4D example, the sequential \added[id=TC]{quadratic} SOS \added[id=TC]{approach} needs significantly less iterations \added[id=TC]{than} CD. The overall solve time is lower due to the low number of iterations. The sequential \added[id=TC]{quadratic SOS approach} was able to solve both problems to a smaller cost. For the CD, we were not able to find an optimal sublevel set $\beta$ \added[id=TC]{with} coordinate descent. \deleted[id=TC]{It either turned into infeasibility  or was computed to zero by the underlying solver.} Thus, for this \added[id=TC]{example} we set $\beta = 0.1$ (larger values lead to infeasibility), a feasible but potentially conservative value. The search for the sublevel set could be carried out by an additional bisection which \added[id=TC]{would result} in a higher computational effort.

\subsection{Compatible Control Barrier and Control Lyapunov Function Synthesis}
\added[id=TC]{We consider} the synthesis of Control Barrier Function (CBF) and a \added[id=TC]{compatible} Control Lyapunov Function (CLF). \added[id=TC]{Here, {\em compatible}} means both conditions \added[id=TC]{can be} satisfied simultaneously, which is \added[id=TC]{shown} by constructing  a \added[id=TC]{shared} control law that \added[id=TC]{for} all CBF and the CLF \added[id=TC]{conditions} \cite{schneeberger_sos_2023}. \added[id=TC]{Adapting} from~\cite{schneeberger_sos_2023} by also incorporating control constraints, 
\added[id=TC]{we solve} the nonconvex SOS problem
\begin{subequations}
    \label{eq: SOSprob1}
\begin{alignat}{2}
    &\min_{\substack{s_1, s_2, s_4 \in\Sigma[x], s_3 \in\Sigma[x]^p \\  \hat h , \hat V \in \mathbb{R}[x], \hat\kappa \in \mathbb{R}[x]^m} } \; \|g - h\|^2_{\mathbb R[x]}   \\
    &\text{s.t. }  \hat V - \varepsilon (x^\top x)                &  & \in\Sigma[x]^{}  \label{eq: strictPos}\\
    &\hphantom{\text{s.t. }}   s_1 (\hat h-\beta) - g                                     &  & \in \Sigma[x]^{} \\
    &\hphantom{\text{s.t. }} s_2 (\hat h-\beta) -  \dot{\hat{h}} - \gamma_h ( \hat{h} - \beta )&  & \in \Sigma[x]^{} \label{eq: CBFSOS}\\
    &\hphantom{\text{s.t. }}  s_3 (\hat h-\beta) - (H_\mathcal{U} \hat\kappa - \mathbf 1_p)                 & & \in \Sigma[x]^p \label{eq: lowerContSOS} \\
    &\hphantom{\text{s.t. }}  s_4 (\hat h-\beta) - \tau                               &  & \in\Sigma[x]^{} \label{eq: terminalPenSOS}
\end{alignat}
\end{subequations}
with $\dot {\hat h} = \langle \nabla \hat h, f(x,\kappa) \rangle$ and  $\tau = \ip{ \nabla \hat V(x), f(x,\hat\kappa(x))}   + \gamma_VV$, where $\gamma_h$ and $\gamma_V$ are extended \added[id=TC]{class-}$\mathcal K$ functions of the form $a(s) = a \, s$ with $a > 0$.
The \added[id=TC]{example} deals with constrained satellite control, with one CBF and one CLF, using the same parameters as \added[id=TC]{in} Subsection~\ref{subsubsec: caseII_sat}. Unlike the previous \added[id=TC]{examples}, we concentrate on the computational aspects of the sequential \added[id=TC]{quadratic} SOS approach and do not compare \added[id=TC]{it} to CD.  The problem has 1756 constraints and 9404 decision variables.
\added[id=TC]{The} sequential SOS approach needs 53 iterations to find an
optimal solution with optimal cost $J(\cdot) = 0.04$. A
detailed breakdown of the computational effort is provided
in Table~\ref{tab:CBFCLF}. The largest portion is to solve the underlying
quadratic subproblem (\SI{67}{\percent}), followed by the filter line search
(\SI{31.8}{\percent}). The high computational effort of the line search can be
primarily traced back to the constraint violation check using
the signed-distance approach (\SI{69}{\percent}). Other components
such as Hessian approximation or function evaluations are negligible compared to the SDPs.

\begin{table}[h!]
    \caption{Timing results for the CBF-CLF Satellite Example in seconds.}
    \label{tab:CBFCLF}
    \centering
    \begin{tabular}{c c  c c } 
       \hline
       \hline
        \textbf{Solve time (Alg.~\ref{alg: fullSQP})} & \textbf{Prob.~(\ref{eq:QsubProb})}   & \textbf{Alg.~\ref{alg: filterLineSearch}}  & \textbf{Others} \\
        68.7 & 46.1 &   21.9 &  0.7 \\
         \hline
  \hline
    \end{tabular}
  \end{table}

\subsection{Discussion and Practical Consideration\added[id=TC]{s}}
The sequential \added[id=TC]{quadratic} SOS \added[id=TC]{approach} with its filter line search allows to directly deal with nonconvex, nonlinear SOS problems, even with noncontrol-affine polynomial system dynamics. Unlike coordinate descent, there is no requirement that the initial guess is a feasible solution\replaced[id=TC]{ and}{. Coordinate descent has in general no convergence guarantees, but} local convergence guarantees can be given for the sequential \added[id=TC]{quadratic SOS} approach under mild assumptions. If the problem is infeasible, confirmed by the feasibility restoration, the user might use these information to adapt the problem \added[id=TC]{and change} the polynomial structure or revise the problem formulation. The feasibility restoration adds robustness to the overall algorithm, as demonstrated by the $N$-link robot arm example. However, \added[id=JO]{there are no guarantees for the overall algorithm or feasibility restoration to find a feasible point~\cite{wachterImplementationInteriorpointFilter2006a}}. Whereas a nonfeasible initial guess might be advantageous, a bad initial guess might lead to insufficient progress \added[id=JO]{or even divergence since it is a Newtons-based method~\cite{betts2010}}. The filter line search alleviates \added[id=TC]{this} problem, but does not fully resolve it. \added[id=TC]{Some} strategies for initial guess generation are  outlined next.

\begin{enumerate} \item If possible, use information from the problem itself, e.g., linearize the system and compute a linear control law and corresponding for Lyapunov function. 
\item Use a simple initial guess (e.g., all polynomials are initialized with basis one) and solve only for a few iterations. Take the result and either use it as an initial guess or use the information to construct an \added[id=TC]{improved} initial guess. 
\item Solve a simplified version, e.g., remove some of the constraints or relax the constraint bounds. 
\item Remove cost function and try to solve a feasibility problem, or relax the thresholds for feasibility. 
\end{enumerate} 

Whereas sequential SOS is capable of solving nonconvex problems in a more direct manner and can reduce the computational effort, it is still restricted by the capabilities of the underlying SDP solver, i.e., the computational effort can still be high for large problems. Yet, the proposed approach is one step closer to more tractable methods for nonlinear system analysis and control design using SOS.




\section{Conclusion}
This paper demonstrates significant computational savings for nonconvex sum-of-squares problems arising in system analysis and control design by leveraging a \added[id=TC]{sequential quadratic programming approach with} filter line search algorithm. In several real-world-inspired case studies, it is demonstrated that the \added[id=TC]{sequential SOS} approach outperforms state-of-the-art methods, thereby making a step towards more practical synthesis and analysis tools for \added[id=TC]{control systems}. \added[id=TC]{The results of} comprehensive benchmarks \added[id=TC]{are theoretically backed by local superlinear or even quadratic convergence rates under regularity and continuity assumptions.} Besides, this paper is accommodated with an open-source implementation.

\begin{acknowledgements}
The authors thank Benoît Legat for his input for the paper and Renato Loureiro for his help to improve the manuscript.
\end{acknowledgements}

\begin{funding}
    This research is partially supported by the Ministry of Science, Research
and Arts of the state of Baden-Württemberg under funding number MWK32-
7531-49/13/7 for the project DaSO: Data-driven Spacecraft Operations.
\end{funding}

\begin{sourcode}
    The open-source implementation of the proposed filter line search algorithm is available as part of CaΣoS: \url{https://github.com/iFR-OFC/casos}.
\end{sourcode}

\begin{dataavail}
    The data regarding the benchmarks are available in \cite{DARUS-5677_2026}.
\end{dataavail}

%

 \section*{Declarations}
\begin{confint}
     The authors declare that they have no conflict of interest.
\end{confint}

\bibliographystyle{spphys}       
\bibliography{OptLit}   

\end{document}